%% file: main.tex
\PassOptionsToPackage{unicode,pdfusetitle}{hyperref}
\PassOptionsToPackage{hyphens}{url}
\PassOptionsToPackage{dvipsnames,svgnames,x11names}{xcolor}

\documentclass[twoside]{article}

\usepackage[accepted]{aistats2023}
\input{preamble}

\begin{document}

%

%
\runningauthor{Johan Larsson, Quentin Klopfenstein, Mathurin Massias, Jonas Wallin}

\twocolumn[%
  \aistatstitle{Coordinate Descent for SLOPE}
  \aistatsauthor{Johan Larsson \And Quentin Klopfenstein}
  \aistatsaddress{
    Department of Statistics\\Lund University, Sweden\\
    \href{mailto:johan.larsson@stat.lu.se}{\url{johan.larsson@stat.lu.se}}
    \And
    Luxembourg Centre for Systems Biomedicine\\ University of Luxembourg, Luxembourg \\
    \href{mailto:quentin.klopfenstein@uni.lu}{\url{quentin.klopfenstein@uni.lu}}
  }
  \aistatsauthor{Mathurin Massias \And Jonas Wallin}
  \aistatsaddress{
    Univ. Lyon, Inria, CNRS, ENS de Lyon,
    \\
    UCB Lyon 1, LIP UMR 5668, F-69342 \\
    Lyon, France \\
    \href{mailto:mathurin.massias@inria.fr}{\url{mathurin.massias@inria.fr}}
    \And
    Department of Statistics\\Lund University, Sweden\\
    \href{mailto:jonas.wallin@stat.lu.se}{\url{jonas.wallin@stat.lu.se}}
  }
]

\begin{abstract}
  \input{abstract}
\end{abstract}

\input{introduction}

\input{notation}

\input{theory}

\input{experiments}
\input{discussion}
\input{acknowledgements}

\printbibliography[heading=subbibliography]

\onecolumn

\appendix

\aistatstitle{Supplement to \emph{Coordinate Descent for SLOPE}}

\input{proofs}
\input{additional_experiments}
\input{other_datafits}
\input{solver_details}

\input{dataset_sources}

\end{document}

%% file: preamble.tex
\usepackage{lmodern}
\usepackage{amssymb,amsmath,amsthm,mathtools,isomath}

\usepackage[T1]{fontenc}
\usepackage[utf8]{inputenc}
\usepackage{textcomp} 
\usepackage[english]{babel}

\usepackage{upquote} 
\usepackage{nicefrac}	
\usepackage[]{microtype}
\UseMicrotypeSet[protrusion]{basicmath} 

\usepackage{stmaryrd}
\usepackage{xcolor}
\usepackage{xurl}
\usepackage{bookmark}
\usepackage{csquotes}
\usepackage{siunitx}
\usepackage{enumitem}

\usepackage{algorithm,algpseudocode}
\usepackage[titlenumbered,linesnumbered,ruled,noend,algo2e]{algorithm2e}

\usepackage[textsize=scriptsize]{todonotes}

\usepackage[noabbrev,nameinlink]{cleveref}

\usepackage{aliascnt}
\newaliascnt{problem}{equation}
\aliascntresetthe{problem}
\creflabelformat{problem}{#2\textup{(#1)}#3}
\makeatletter

\def\endproblem{\eqno \hbox{\@eqnnum}$$\@ignoretrue}
\makeatother
\Crefname{problem}{Problem}{Problems}

\usepackage{hyperref}
\hypersetup{
  colorlinks = true,
  linkcolor  = RoyalBlue4,
  filecolor  = RoyalBlue4,
  citecolor  = VioletRed4,
  urlcolor   = RoyalBlue4
}

\usepackage{longtable}
\usepackage{booktabs}

\usepackage{etoolbox}

\usepackage{footnotehyper}
\makesavenoteenv{longtable}

\usepackage{graphicx}
\graphicspath{{figures/}}
\usepackage{subcaption}

\usepackage[style=authoryear,language=english,uniquelist=false,maxbibnames=10,minbibnames=5,uniquename=false,doi=false]{biblatex}
\DeclareMathOperator*{\argmin}{arg\,min}

\DeclareMathOperator{\sign}{sign}

\DeclareMathOperator{\prox}{prox}
\DeclareMathOperator{\Span}{span}
\DeclareMathOperator{\cumsum}{cumsum}
\DeclareMathOperator{\parset}{par}

\newcommand{\norm}[1]{\lVert {#1} \rVert}
\newcommand{\bbR}{\mathbb{R}}
\newcommand{\cB}{\mathcal{B}}
\newcommand{\cC}{\mathcal{C}}
\newcommand{\cG}{\mathcal{G}}
\newcommand{\cM}{\mathcal{M}}
\newcommand{\pkg}[1]{\textsf{#1}}
\newcommand{\dataset}[1]{\texttt{#1}}

\theoremstyle{plain}
\newtheorem{theorem}{Theorem}[section]
\newtheorem{proposition}[theorem]{Proposition}
\newtheorem{lemma}[theorem]{Lemma}

\theoremstyle{definition}
\newtheorem{definition}[theorem]{Definition}

\theoremstyle{remark}
\newtheorem{remark}[theorem]{Remark}

%% file: abstract.tex
The lasso is the most famous sparse regression and feature selection method. 
One reason for its popularity is the speed at which the underlying optimization problem can be solved. 
Sorted L-One Penalized Estimation (SLOPE) is a generalization of the lasso with appealing statistical properties.
In spite of this, the method has not yet reached widespread interest.
A major reason for this is that current software packages that fit SLOPE rely on algorithms that perform poorly in high dimensions. 
To tackle this issue, we propose a new fast algorithm to solve the SLOPE optimization problem,
which combines proximal gradient descent and proximal coordinate descent steps.
We provide new results on the directional derivative of the SLOPE penalty and its related SLOPE thresholding operator, as well as provide convergence guarantees for our proposed solver.
In extensive benchmarks on simulated and real data, we show that our method outperforms a long list of competing algorithms.

%% file: introduction.tex
\section{INTRODUCTION}\label{sec:introduction}

In this paper we present a novel numerical algorithm for Sorted L-One Penalized
Estimation (SLOPE, \cite{bogdan2013,bogdan2015,zeng2014ordered}), which, for a
design matrix \(X \in \mathbb{R}^{n \times p}\) and response vector \(y \in \mathbb{R}^n\), is defined as
\begin{problem}\label{pb:slope}
  \min_{\beta \in \mathbb{R}^p}
  P(\beta) =  \frac{1}{2} \norm{y - X \beta}^2 + J(\beta)
\end{problem}
where
\begin{equation}
  \label{eq:sorted-l1-norm}
  J(\beta) = \sum_{j=1}^p \lambda_j|\beta_{(j)}|
\end{equation}
is the \emph{sorted \(\ell_1\) norm}, defined through
\begin{equation}
  |\beta_{(1)}| \geq |\beta_{(2)}| \geq \cdots \geq |\beta_{(p)}| \enspace,
\end{equation}
with \(\lambda\) being a fixed non-increasing and non-negative sequence.

The sorted $\ell_1$ norm is a sparsity-enforcing penalty that has become
increasingly popular due to several appealing properties, such as its ability
to control false discovery rate~\parencite{bogdan2015,kos2020}, cluster
coefficients~\parencite{figueiredo2016, schneider2020a}, and recover sparsity and
ordering patterns in the solution~\parencite{bogdan2022}. Unlike other competing
sparse regularization methods such as MCP~\parencite{zhang2010} and
SCAD~\parencite{fan2001}, SLOPE has the advantage of being a convex problem~\parencite{bogdan2015}.

In spite of the availability of predictor screening
rules~\parencite{larsson2020c,elvira2022}, which help speed up SLOPE in the
high-dimensional regime, current state-of-the-art algorithms for SLOPE perform
poorly in comparison to those of more established penalization methods such as
the lasso (\(\ell_1\) norm regularization) and ridge regression
(\(\ell_2\) norm regularization).
As a small illustration of this issue, we compared the speed at which the \pkg{SLOPE}~\parencite{larsson2022d} and \pkg{glmnet}~\parencite{friedman2022} packages solve a SLOPE and lasso problem, respectively, for the \dataset{bcTCGA} data set.
\pkg{SLOPE} takes 43 seconds to reach convergence, whilst \pkg{glmnet} requires only 0.14 seconds\footnote{See~\Cref{sec:slope-vs-glmnet} for details on this experiment.}.
This lackluster performance has hampered the applicability of SLOPE to many real-world applications.
In this paper we present a remedy for this issue, by presenting an algorithm that reaches convergence in only 2.9 seconds on the same problem\footnote{Note that we do not use any screening rule in the current implementation of our algorithm, unlike the \pkg{SLOPE} package, which uses the strong screening rule for SLOPE~\parencite{larsson2020c}.}.

A major reason for why algorithms for solving
$\ell_1$-, MCP-, or SCAD-regularized problems enjoy better performance is that
they use coordinate
descent~\parencite{tseng2001convergence,friedman2010,breheny2011}. Current SLOPE
solvers, on the other hand, rely on proximal gradient descent algorithms such
as FISTA~\parencite{beck2009} and the alternating direction method of multipliers
method (ADMM, \cite{boyd2010}), which have proven to be less efficient than
coordinate descent in empirical benchmarks on related problems, such as the
lasso~\parencite{moreau2022benchopt}.
In addition to FISTA and ADMM, there has also been research into Newton-based augmented Lagrangian methods to solve SLOPE~\parencite{Ziyan2019}.
But this method is adapted only to the \(p \gg n\) regime and, as we show in our paper, is outperformed by our method even in this scenario.
Applying coordinate descent to SLOPE is not,
however, straightforward since convergence guarantees for coordinate descent
require the non-smooth part of the objective to be separable, which is not the case for SLOPE. As a
result, naive coordinate descent schemes can get
stuck~(\Cref{fig:naive-cd-stuck}).

\begin{figure}[htb]
  \centering
  \includegraphics[]{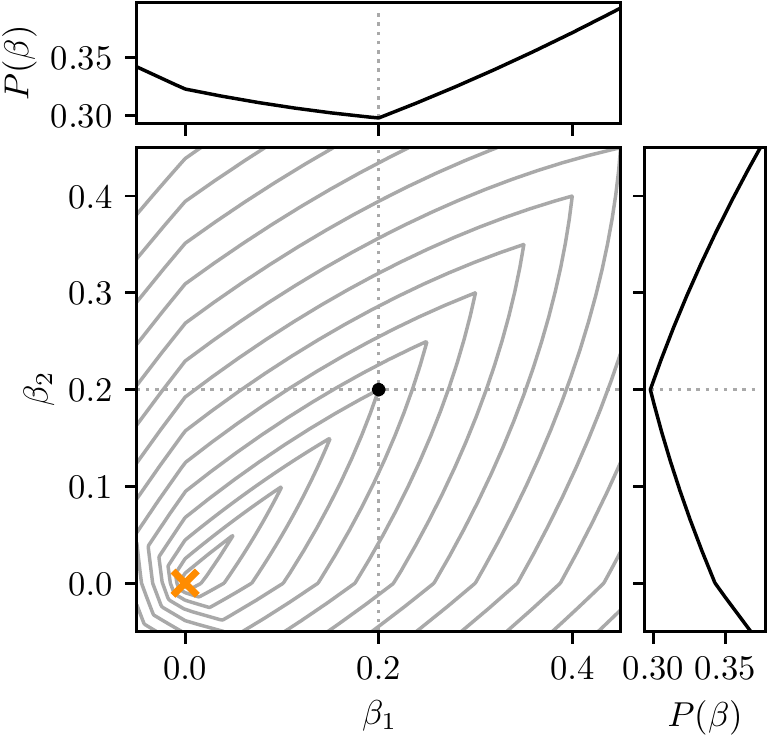}
  \caption{%
  An example of standard coordinate descent getting stuck on a two-dimensional SLOPE problem.
  The main plot shows level curves for the primal objective~\eqref{pb:slope}, with the minimizer \(\beta^* = [0, 0]^T\) indicated by the orange cross.
  The marginal plots display objective values at \(\beta_1 = 0.2\) when optimizing over \(\beta_2\) and vice versa.
  At \(\beta = [0.2,0.2]^T\), standard coordinate descent can only move in the directions indicated by the dashed lines---neither of which are descent directions for the objective.
  As a result, the algorithm is stuck at a suboptimal point.
  }
  \label{fig:naive-cd-stuck}
\end{figure}

In this article we address this problem by introducing a new, highly effective
algorithm for SLOPE based on a hybrid proximal gradient and coordinate descent
scheme. Our method features convergence guarantees and reduces the time
required to fit SLOPE by orders of magnitude in our empirical experiments.

%% file: notation.tex
\paragraph{Notation}\label{sec:notation}

Let \((i)^{-}\) be the inverse of \((i)\) such that
\(\big((i)^-\big)^- = (i)\); see \Cref{tab:permutation-example} for an
example of this operator for a particular \(\beta\).
\begin{table}[bt]
  \centering
  \caption{Example of the permutation operator \((i)\) and its inverse
    \((i)^-\) for $\beta = [0.5, -5, 4]^T$}
    \label{tab:permutation-example}
  \begin{tabular}{cS[table-format=-1.1,round-mode=off]cc}
    \toprule
    \(i\) & {\(\beta_i\)} & \((i)\) & \((i)^-\) \\
    \midrule
    1     & 0.5         & 2       & 3         \\
    2     & -5          & 3       & 1         \\
    3     & 4           & 1       & 2         \\
    \bottomrule
  \end{tabular}
\end{table}
This means that
\[
  J(\beta) = \sum_{j=1}^p \lambda_j |\beta_{(j)}|
  = \sum_{j=1}^p \lambda_{(j)^-}|\beta_j| \,.
\]

Sorted $\ell_1$ norm penalization leads to solution vectors with clustered coefficients in which the absolute values of several coefficients are set to exactly the same value.
To this end, for a fixed $\beta$ such that $|\beta_j|$ takes $m$ distinct values, we introduce \(\mathcal{C}_1, \mathcal{C}_2 , \dots, \mathcal{C}_m\) and \(c_1,
c_2, \dots, c_m\) for the indices and coefficients respectively of the \(m\)
clusters of $\beta$, such that
$\mathcal{C}_i = \{j : |\beta_j| = c_i\}$ and $c_1 > c_2 > \cdots > c_m \geq 0.$
For a set $\cC$, let \(\bar{\mathcal{C}}\) denote its complement.
Furthermore, let $(e_i)_{i \in [d]}$ denote the canonical basis of $\bbR^d$, with \([d] = \{1,2,\dots,d\}\).
Let $X_{i:}$ and $X_{:i}$ denote the $i$-th row and column, respectively, of the matrix $X$.
Finally, let $\sign(x) = x / |x|$ (with the convention 0/0 = 1) be the scalar sign, that acts entrywise on vectors.

%% file: theory.tex

\section{COORDINATE DESCENT FOR SLOPE}\label{sec:theory}

Proximal coordinate descent cannot be applied to \Cref{pb:slope} because the non-smooth term is not separable.
If the clusters $\mathcal{C}_1^*, \ldots, \mathcal{C}_{m^*}^*$ and signs of the solution $\beta^*$ were known, however, then the values $c_1^*, \ldots, c_{m^*}^*$ taken by the clusters of $\beta^*$ could be computed by solving
\begin{problem}\label{pb:separable_slope}
\begin{multlined}
  \min_{z \in \bbR^{m^*}}\bigg(
  \frac{1}{2} \Big\lVert y - X \sum_{i=1}^{m^*} \sum_{j \in \mathcal{C}_i^*} z_i \sign(\beta_j^*) e_j \Big\rVert^2 \\
  + \sum_{i=1}^{m^*} | z_i | \sum_{j \in \mathcal{C}_i^*} \lambda_j
  \bigg).
\end{multlined}
\end{problem}
Conditionally on the knowledge of the clusters and the signs of the coefficients, the penalty becomes separable~\parencite{dupuis2021}, which means that coordinate descent could be used.

Based on this idea, we derive a coordinate descent update for minimizing the SLOPE problem~\eqref{pb:slope} with respect to the coefficients of a single cluster at a time~(\Cref{sec:cd-update}).
Because this update is limited to updating and, possibly, merging clusters, we intertwine it with proximal gradient descent in order to correctly identify the clusters~(\Cref{sec:pgd-update}).
In \Cref{sec:hybrid-strategy}, we present this hybrid strategy and show that is guaranteed to converge.
In \Cref{sec:experiments}, we show empirically that our algorithm outperforms competing alternatives for a wide range of problems.

\subsection{Coordinate Descent Update}
\label{sec:cd-update}

In the sequel, let $\beta$ be fixed with $m$ clusters $\mathcal{C}_1, \ldots, \mathcal{C}_m$ corresponding to values $c_1, \ldots, c_m$.
In addition, let $k \in [m]$ be fixed and $s_k = \sign \beta_{\mathcal{C}_k}$.
We are interested in updating $\beta$ by changing only the value taken on the $k$-th cluster.
To this end, we define $\beta(z) \in \bbR^p$ by:
\begin{equation}
  \label{eq:coordinate-update-beta}
  \beta_i(z) =
  \begin{cases}
    \mathrm{sign}(\beta_i) z   \, , & \text{if } i \in \mathcal{C}_k \, , \\
    \beta_i \, ,                    & \text{otherwise} \, .
  \end{cases}
\end{equation}
Minimizing the objective in this direction amounts to solving the following
one-dimensional problem:
\begin{problem}
  \label{pb:cluster-problem}
  \min_{z \in \mathbb{R}} \Big(
  G(z) = P(\beta(z))  = \frac{1}{2} \norm{y - X \beta(z)}^2 + H(z)
  \Big) \,  ,
\end{problem}
where
\begin{equation}
  H(z) = |z| \sum_{j \in \mathcal{C}_k} \lambda_{(j)^-_z}
  + \sum_{j \notin \mathcal{C}_k} |\beta_j| \lambda_{(j)^-_z}
\end{equation}
is the \emph{partial sorted \(\ell_1\) norm} with respect to the \(k\)-th cluster and where we write \(\lambda_{(j)^-_z}\) to indicate that the inverse sorting permutation \((j)^-_z\)
is defined with respect to \(\beta(z)\).
The optimality condition for \Cref{pb:cluster-problem} is
\[
  \forall \delta \in \{-1, 1\}, \quad G'(z; \delta) \geq 0,
\]
where $G'(z; \delta) $ is the directional derivative of $G$ in the direction $\delta$.
Since the first part of the objective is differentiable, we have
\[
  G'(z; \delta)  = \delta \sum_{j \in \mathcal{C}_k} X_{:j}^\top(X\beta(z) - y) + H'(z; \delta) \, ,
\]
where \(H'(z; \delta)\) is the directional derivative of $H$.

Throughout the rest of this section we derive the solution to \eqref{pb:cluster-problem}.
To do so, we will introduce the directional derivative for the
sorted \(\ell_1\) norm with respect to the coefficient of the \(k\)-th cluster.
First, as illustrated on \Cref{fig:partial_slope}, note that $H$ is piecewise affine, with breakpoints at 0 and all $\pm c_i$'s for $i \neq k$.
Hence, the partial derivative is piecewise constant, with jumps at these points; in addition, $H'(\cdot; 1) = H'(\cdot, -1)$ except at these points.

\begin{figure}[htbp]
  \centering
  \includegraphics[width=\linewidth]{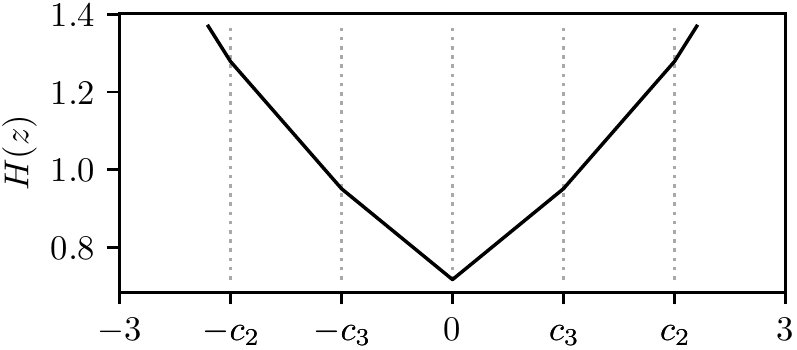}
  \caption{Graph of the partial sorted $\ell_1$ norm with \(\beta = [-3, 1, 3, 2]^T\), \(k = 1\), and so $c_1, c_2, c_3 = (3, 2, 1)$.}
  \label{fig:partial_slope}
\end{figure}

Let \(C(z)\) be the function that returns the cluster of $\beta(z)$ corresponding to \(|z|\), that is
\begin{equation}
  C(z) = \{j : |\beta(z)_j| = |z|\} \,.
\end{equation}

\begin{remark}\label{rem:permutation_C_z}
  Note that if $z$ is equal to some $c_i$, then $C(z) = \mathcal{C}_i \cup \mathcal{C}_k$, and otherwise $C(z) = \mathcal{C}_k$.
  Related to the piecewise affineness of $H$ is the fact that the permutation\footnote{the permutation is in fact not unique, without impact on our results. This is discussed when needed in the proofs.} corresponding to $\beta(z)$ is
  \begin{equation*}
    \begin{cases}
      \cC_k, \cC_m, \ldots, C_1
       & \text{ if } z \in \left]0, c_m\right[ \, ,                                                                     \\
      \cC_m, \ldots ,\cC_i, \cC_k, \cC_{i-1}, \ldots, C_1
       & \splitfrac{\text{ if } z \in \left]c_{i}, c_{i-1} \right[}{\text{ and } i \in \llbracket 2 , m \rrbracket\, ,} \\
      \cC_m, \ldots C_1,  \cC_k
       & \text{ if } z \in \left]c_1, +\infty \right[ \, ,                                                              \\
    \end{cases}
  \end{equation*}
  and that this permutation also reorders $\beta(z \pm h)$ for $z \neq c_i \; (i \neq k)$ and $h$ small enough.
  The only change in permutation happens when $z = 0$ or $z = c_i \; (i \neq k)$.
  Finally, the permutations differ between $\beta(z + h)$ and $\beta(z - h)$ for arbitrarily small $h$ if and only if $z = c_i \neq 0$.
\end{remark}

We can now state the directional derivative of $H$. 

\begin{theorem}\label{thm:sl1-directional-derivative}
  Let \(c^{\setminus k}\) be the set containing all elements of $c$ except the $k$-th one: $c^{\setminus k} =  \{c_1, \ldots c_{k-1}, c_{k+1}, \ldots, c_m \}$.
  Let $\varepsilon_c > 0$ such that
  \begin{equation}
    \label{eq:epsilon-c}
    \varepsilon_c < \big| c_i - c_j\big| , \quad \forall\, i \neq j \text{ and } \varepsilon_c < c_m \text{ if } c_m \neq 0 \, .
  \end{equation}
  The directional derivative of the partial sorted $\ell_1$ norm with respect to the $k$-th cluster, \(H\), in the direction \(\delta\) is
  \[
    H'(z; \delta) =
    \begin{cases}
      \smashoperator[r]{\sum_{j \in C(\varepsilon_c )}} \lambda_{(j)^-_{\varepsilon_c }}
       & \text{if } z = 0 \, ,                                \\
      \sign(z)\delta\smashoperator{\sum_{j \in C(z + {\varepsilon_c} \delta)}} \lambda_{(j)^-_{z + {\varepsilon_c}\delta}}
       & \text{if } |z| \in c^{\setminus k} \setminus \{0\} , \\
      \sign(z)\delta\smashoperator{\sum_{j \in C(z)}} \lambda_{(j)^-_{z}}
       & \text{otherwise} \, .
    \end{cases}
  \]
\end{theorem}
The proof is in \Cref{app:proof_directional_derivative}; in \Cref{fig:directional-derivative}, we show an example of the directional
derivative and the objective function.

\begin{figure}[htb]
  \centering
  \includegraphics[]{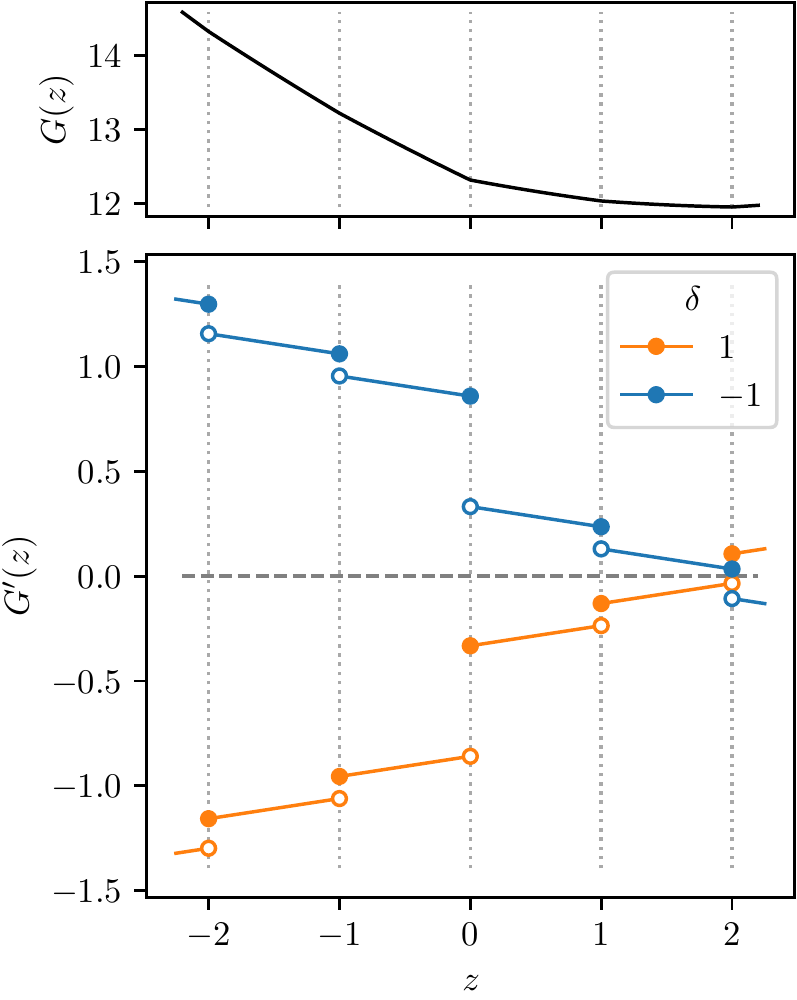}
  \caption{%
  The function \(G\) and its directional derivative \(G'( \cdot ; \delta)\) for
  an example with \(\beta = [-3, 1, 3, 2]^T\), \(k = 1\), and consequently
  \(c^{\setminus k} = \{1, 2\}\). The solution of \Cref{pb:cluster-problem} is the value of \(z\) for
  which \(G'(z; \delta) \geq 0 \) for \(\delta \in \{-1, 1\}\), which holds only
  at \(z = 2\), which must therefore be the solution.
  }
  \label{fig:directional-derivative}
\end{figure}

Using the directional derivative, we can now introduce the SLOPE thresholding operator.

\begin{theorem}[The SLOPE Thresholding Operator]
  \label{thm:thresholding-operator}
  Define \(S(x) = \sum_{j \in C(x)}\lambda_{(j)^-_{x}}\) and
  let
  \[
    \begin{multlined}
      T(\gamma; \omega, c, \lambda) = \\
      \begin{cases}
        0
         & \text{if } |\gamma| \leq S(\varepsilon_c),               \\
        \sign(\gamma)c_i
         & \text{if } \omega c_i + S(c_i - \varepsilon_c)           \\
         & \quad \leq |\gamma| \leq                                 \\
         & \quad \omega c_i + S(c_i + \varepsilon_c),               \\
        \frac{\sign(\gamma)}{\omega} \big( |\gamma| - S(c_i + \varepsilon_c) \big)
         & \text{if } \omega c_i + S(c_i + {\varepsilon_c})         \\
         & \quad < |\gamma| <                                       \\
         & \quad \omega c_{i - 1} + S(c_{i - 1} - {\varepsilon_c}), \\
        \frac{\sign(\gamma)}{\omega} \big( |\gamma| - S(c_1 + {\varepsilon_c}) \big)
         & \text{if } |\gamma| \geq                                 \\
         & \quad \omega c_1 + S(c_1 + {\varepsilon_c}).
      \end{cases}
    \end{multlined}
  \]
  with \({\varepsilon_c}\) defined as in \eqref{eq:epsilon-c}.
  Let $\tilde x = X_{\cC_k} \sign(\beta_{\cC_k})$
  and \(r = y - X\beta\).
  Then
  \begin{equation}
    \begin{multlined}
      T \left(c_k\norm{\tilde x}^2 + \tilde x^Tr; \norm{x}^2, c^{\setminus k}, \lambda \right) = \argmin_{z \in \mathbb{R}} G(z) \,.
    \end{multlined}
  \end{equation}
\end{theorem}
An illustration of this operator is given in \Cref{fig:slope-thresholding}.
\begin{remark}
  The minimizer is unique because \(G\) is the sum of a quadratic function in one variable and a norm.
\end{remark}

\begin{remark}
  In the lasso case where the $\lambda_i$'s are all equal, the SLOPE thresholding operator reduces to the soft thresholding operator.
\end{remark}

In practice, it is rarely necessary to compute all sums in \Cref{thm:thresholding-operator}.
Instead, we first check in which direction we need to search relative to the current order for the cluster and then search in that direction until we find the solution.
The complexity of this operation depends on how far we need to search and the size of the current cluster and other clusters we need to consider.
In practice, the cost is typically larger at the start of optimization and becomes marginal as the algorithm approaches convergence and the cluster permutation stabilizes.

\begin{figure*}[htb]
  \centering
  \includegraphics[]{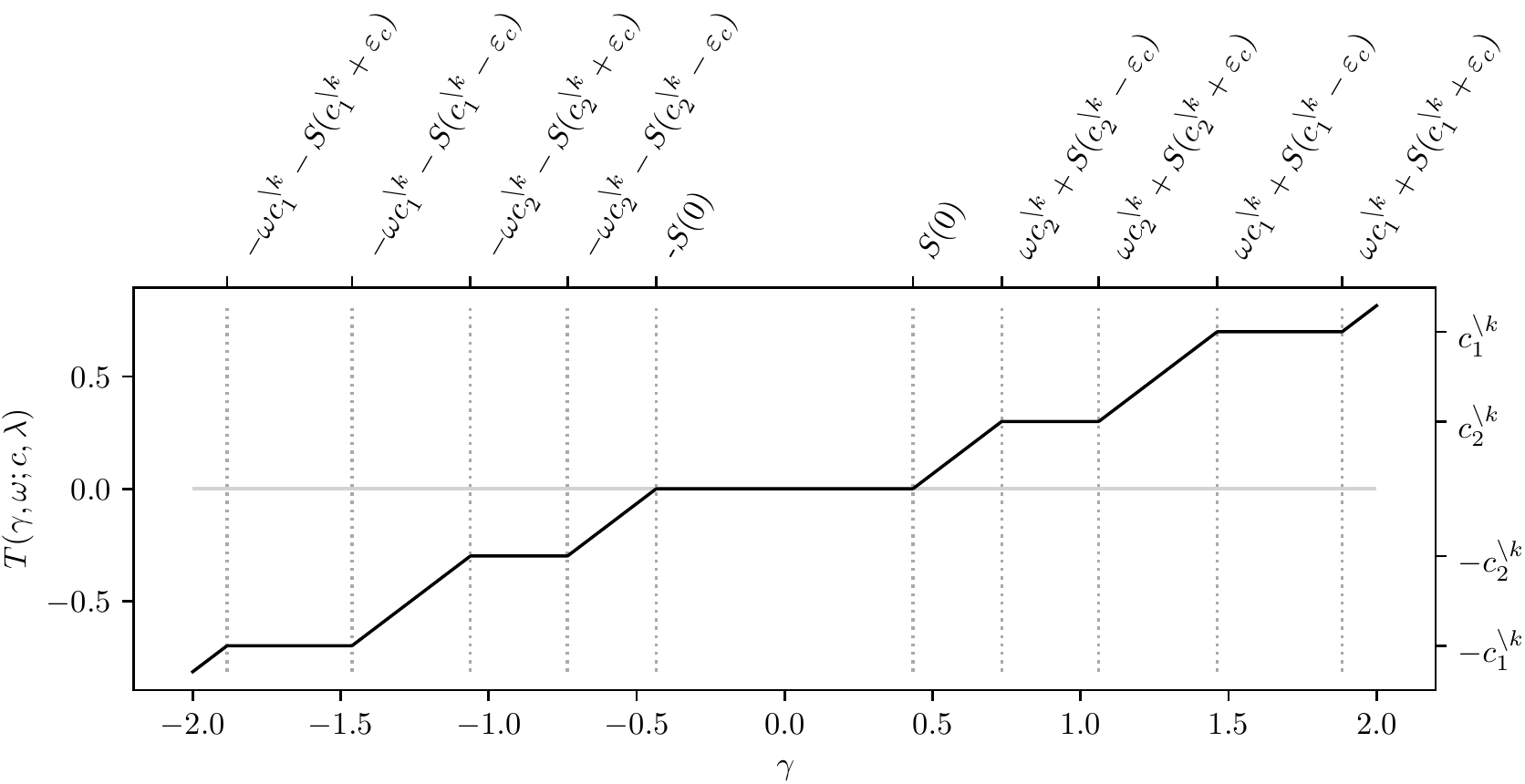}
  \caption{%
  An example of the SLOPE thresholding operator for \(\beta = [0.5, -0.5, 0.3, 0.7]^T\), \(c = (0.7, 0.5, 0.3)\)
  with an update for the second cluster (\(k = 2\)), such that
  \(c^{\setminus k} = (0.5, 0.3)\). Across regions where the function is constant,
  the operator sets the result to be either exactly 0 or to the value of one
  of the elements of \(\pm c^{\setminus k}\).
  }
  \label{fig:slope-thresholding}
\end{figure*}

\subsection{Proximal Gradient Descent Update}
\label{sec:pgd-update}

The coordinate descent update outlined in the previous section updates the coefficients of each cluster in unison, which allows clusters to merge---but not to split.
This means that the coordinate descent updates are not guaranteed to identify the clusters of the solution on their own.
To circumvent this issue, we combine these coordinate descent steps with full proximal gradient steps, which enable the algorithm to identify the cluster structure~\parencite{Liang2014} due to the partial smoothness property of the sorted \(\ell_1\) norm that we prove in \Cref{app:sec:partly_smooth}.
A similar idea has previously been used in \textcite{bareilles2022newton}, wherein Newton steps are taken on the problem structure identified after a proximal gradient descent step. 

\subsection{Hybrid Strategy}
\label{sec:hybrid-strategy}

We now present the proposed solver in \Cref{alg:hybrid}.
For the first and every $v$-th iteration\footnote{Our experiments suggest that \(v\) has little impact on performance as long as it is at least 3~(\Cref{sec:pgd-freq-study}). We have therefore set it to 5 in our experiments.}, we perform a proximal gradient descent update.
For the remaining iterations, we take coordinate descent steps.

\begin{algorithm}[hbt]
  \SetKwInOut{Input}{input}
  \caption{%
    Hybrid coordinate descent and proximal gradient descent algorithm
    for SLOPE\label{alg:hybrid}}
  \Input{%
    \(X \in \mathbb{R}^{n\times p}\),
    \(y\in \mathbb{R}^n\),
    \(\lambda \in \{\mathbb{R}^p : \lambda_1 \geq \lambda_2 \geq \cdots > 0\}\),
    \(v \in \mathbb{N}\),
    \(\beta \in \mathbb{R}^p\)
  }

  \For{\(t \gets 0,1,\dots\)}{

    \If{\(t \bmod v = 0\)}{
      \(\beta \leftarrow \prox_{J/{\norm{X}^2_2}}\Big(\beta - \frac{1}{\norm{X}_2^2}X^T(X \beta - y)\Big)\) \label{alg:hybrid-istastep}

      Update \(c\), \(\mathcal{C}\)
    }
    \Else{
      \(k \gets 1\)

      \While{\(k \leq \lvert \mathcal{C} \rvert\)}{
        \(\tilde x_k \gets X_{\mathcal{C}_k} \sign(\beta_{\cC_k}) \)

        \(z \gets T(c_k\norm{\tilde x}^2 - \tilde x^T(X\beta - y); \norm{x}^2, c^{\setminus k}, \lambda)\)

        $\beta_{\cC_k} \gets z \sign(\beta_{\cC_k})$

        Update \(c\), \(\mathcal{C}\)

        \(k \gets k + 1\)
      }
    }
  }
  \Return{\(\beta\)}
\end{algorithm}

The combination of the proximal gradient steps and proximal coordinate descent allows us to overcome the problem of vanilla proximal coordinate descent getting stuck because of non-separability and allows us to enjoy the speed-up provided by making local updates on each cluster, as we illustrate in \Cref{fig:illustration-solver}.

\begin{figure*}[htb]
  \centering
  \includegraphics{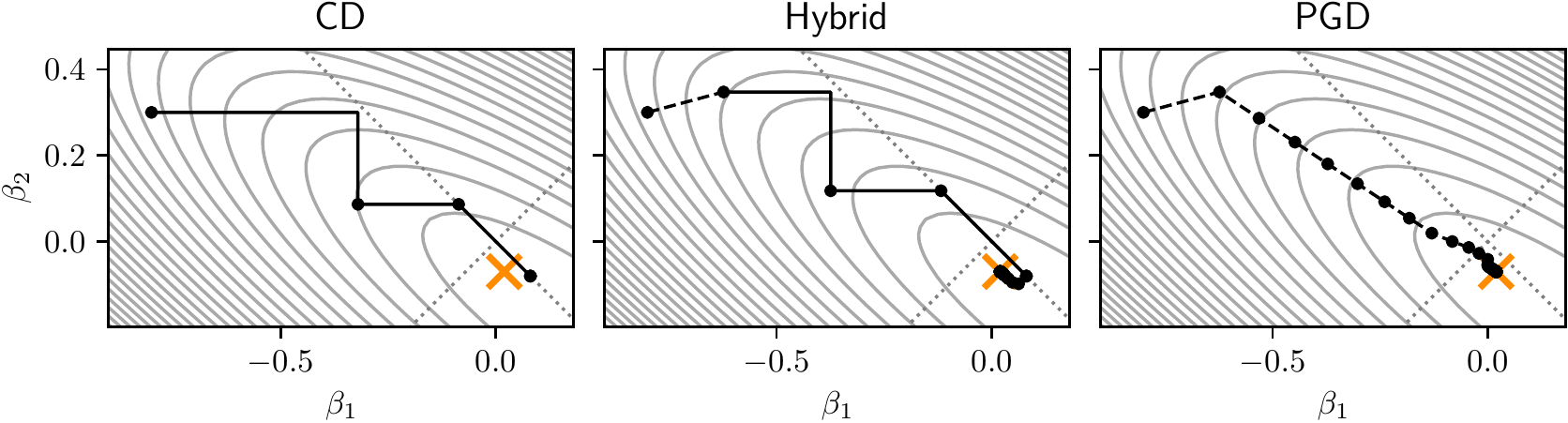}
  \caption{Illustration of the proposed solver. The figures show progress
    until convergence for the coordinate descent (CD) solver that we use as part
    of the hybrid method, our hybrid method, and  proximal gradient descent
    (PGD). The orange cross marks the optimum. Dotted lines indicate where the
    coefficients are equal in absolute value. The dashed lines indicate PGD
    steps and solid lines CD steps. Each dot marks a complete epoch, which may
    correspond to only a single coefficient update for the CD and hybrid
    solvers if the coefficients flip order. Each solver was run until the duality
    gap was smaller than \(10^{-10}\). Note that the CD algorithm cannot split clusters
    and is therefore stuck after the third epoch. The hybrid and PGD algorithms,
    meanwhile, reach convergence after 67 and 156 epochs respectively.}
  \label{fig:illustration-solver}
\end{figure*}

We now state that our proposed hybrid algorithm converges to a solution of \Cref{pb:slope}.


\begin{lemma}
  \label{lem:convergence}
  Let \(\beta^{(t)}\) be an iterate generated by \Cref{alg:hybrid}. Then
  \[
    \lim_{t \rightarrow \infty}\big(P(\beta^{(t)}) - P(\beta^*)\big) = 0.
  \]
\end{lemma}

\paragraph{Alternative Datafits}

So far we have only considered sorted \(\ell_1\)-penalized least squares regression.
In \Cref{sec:other-datafits}, we consider possible extensions to alternative datafits.

%% file: experiments.tex
\section{EXPERIMENTS}\label{sec:experiments}

To investigate the performance of our algorithm, we performed an extensive benchmark against the following competitors:
\begin{itemize}[noitemsep]
  \item Alternating direction method of multipliers (\texttt{ADMM}, \cite{boyd2010}).
        We considered several alternative for the choice of the augmented Lagragian parameter: an adaptive method to update the parameter throughout the algorithm~\parencite[Sec. 3.4.1]{boyd2010} and fixed values.
        In the following sections, we only kept the \texttt{ADMM} solver with a fixed value of $100$ for the augmented Lagrangian parameter.
        We present in \Cref{sec:admm-benchmarks} a more detailed benchmarks for \texttt{ADMM} solvers with different values of this parameter and the adaptive setting.
        Choosing this parameter is not straightforward and the best value changes across datasets and regularization strengths.
  \item Anderson acceleration for proximal gradient descent (\texttt{Anderson PGD}, \cite{zhang2020})
  \item Proximal gradient descent (\texttt{PGD}, \cite{combettes2005})
  \item Fast Iterative Shrinkage-Thresholding Algorithm (\texttt{FISTA}, \cite{beck2009})
  \item Semismooth Newton-Based Augmented Lagrangian (\texttt{Newt-ALM}, \cite{Ziyan2019})

  \item The hybrid (our) solver (see \Cref{alg:hybrid}) combines proximal gradient descent
        and coordinate descent to overcome the non-separability of the SLOPE problem.
  \item The oracle solver (\texttt{oracle CD}) solves \Cref{pb:separable_slope} with coordinate descent, using the clusters obtained via another solver.
        Note that it cannot be used in practice as it requires knowledge of the solution's clusters.
\end{itemize}

We used \pkg{Benchopt}~\parencite{moreau2022benchopt} to obtain the convergence curves for the different solvers.
\pkg{Benchopt} is a collaborative framework that allows reproducible and automatic benchmarks.
The repository to reproduce the benchmark is available at \href{https://github.com/klopfe/benchmark\_slope}{\url{github.com/klopfe/benchmark\_slope}}.

Unless we note otherwise, we used the Benjamini--Hochberg method to compute the \(\lambda\) sequence~\parencite{bogdan2015},
which sets $\lambda_j = \eta^{-1}(1 - q\times j / (2p))$ for $j=1, 2, \hdots, p$ where $\eta^{-1}$ is the probit function.
For the rest of the experiments section, the parameter $q$ of this sequence has been set to $0.1$ if not stated otherwise.\footnote{We initially experimented with various settings for \(q\) but found that they made little difference to the relative performance of the algorithms.}
We let \(\lambda_\text{max}\) be the \(\lambda\) sequence such that \(\beta^* = 0\), but for which any scaling with a strictly positive scalar smaller than one produces a solution with at least one non-zero coefficient.
We then parameterize the experiments by scaling \(\lambda_\text{max}\), using the fixed factors \(1/2\), \(1/10\), and \(1/50\), which together cover the range of very sparse solutions to the almost-saturated case.

We pre-process datasets by first removing features with less than three non-zero values. Then, for dense data we center and scale each feature by its mean and standard deviation respectively.
For sparse data, we scale each feature by its maximum absolute value.

Each solver was coded in \pkg{python}, using \pkg{numpy}~\parencite{harris2020} and \pkg{numba}~\parencite{lam2015} for performance-critical code.
The code is available at \href{https://github.com/jolars/slopecd}{\url{github.com/jolars/slopecd}}.
In \Cref{sec:solver-details}, we provide additional details on the implementations of some of the solvers used in our benchmarks.

The computations were carried out on a computing cluster with dual Intel Xeon CPUs (28 cores) and 128 GB of RAM.

\subsection{Simulated Data}
\label{sec:experiments-simulated-data}

\begin{figure*}[!t]
  \centering
  \includegraphics{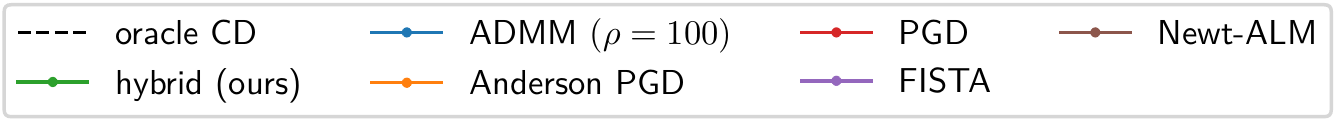}
  \includegraphics{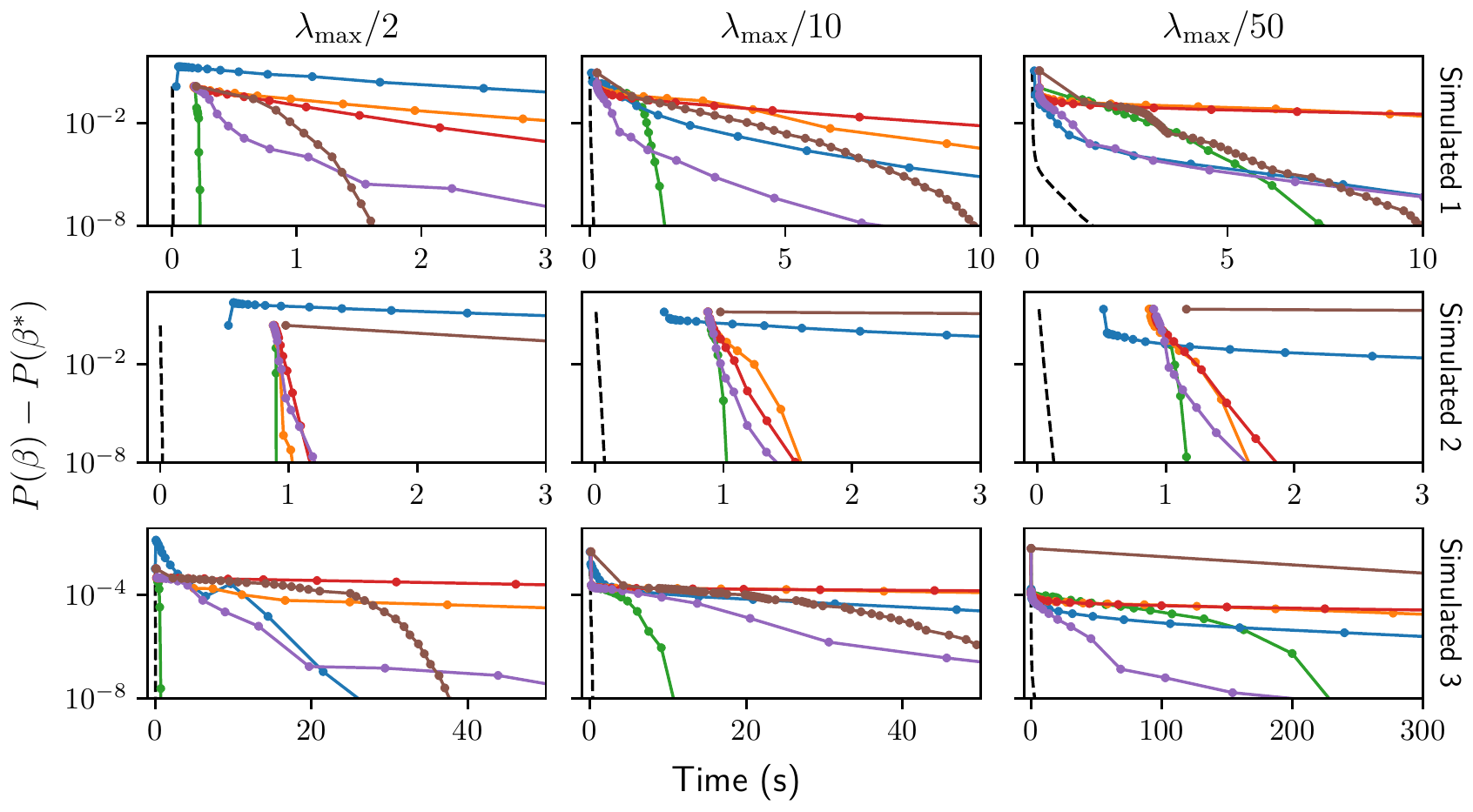} \caption{Benchmark on simulated datasets. The plots show suboptimality as a function of time for SLOPE on multiple simulated datasets and $\lambda$ sequences of varying strength.}
  \label{fig:simulated}
\end{figure*}

The design matrix $X$ was generated such that features had mean one and unit variance, with correlation between features $j$ and $j'$ equal to $0.6^{|j-j'|}$.
We generated \(\beta \in \mathbb{R}^p\) such that \(k\) entries, chosen uniformly at random throughout the vector, were sampled from a standard Gaussian distribution.
The response vector, meanwhile, was set to $y=X\beta + \varepsilon$, where
$\varepsilon$ was sampled from a multivariate Gaussian distribution with variance such that $\lVert X\beta\rVert / \lVert \varepsilon \rVert = 3$.
The different scenarios for the simulated data are described in \Cref{tab:simulated-data}.

\begin{table}[hbt]
  \centering
  \caption{Scenarios for the simulated data in our benchmarks}
  \label{tab:simulated-data}
  \begin{tabular}{
      l
      S[table-format=5.0,round-mode=off]
      S[table-format=7.0,round-mode=off]
      S[table-format=2.0,round-mode=off]
      S[table-format=1.3,round-mode=off]
    }
    \toprule
    {Scenario} & {\(n\)} & {\(p\)} & {\(k\)} & {Density} \\ \midrule
    1          & 200     & 20000   & 20      & 1         \\
    2          & 20000   & 200     & 40      & 1         \\
    3          & 200     & 200000  & 20      & 0.001     \\ \bottomrule
  \end{tabular}
\end{table}

In \Cref{fig:simulated}, we present the results of the benchmarks on simulated data.
We see that for smaller fractions of $\lambda_{\text{max}}$ our hybrid algorithm allows significant speedup in comparison to its competitors mainly when the number of features is larger than the number of samples.
On very large scale data such as in simulated data setting $3$, we see that the hybrid solver is faster than its competitors by one or two orders of magnitude.

For the second scenario, notice that all solvers take considerably longer than the \texttt{oracle CD} method to reach convergence.
This gap is a consequence of Cholesky factorization in the case of \texttt{ADMM} and computation of \(\norm{X}_2\) in the remaining cases.
For the hybrid method, we can avoid this cost, with little impact on performance, since \(\norm{X}_2\) is used only in the PGD step.

\subsection{Real data}
\label{sec:experiments-real-data}

The datasets used for the experiments have been described in \Cref{tab:real-data} and were obtained from \textcite{chang2011,chang2016,breheny2022}.

\begin{table}[hbt]
  \centering
  \caption{%
    List of real datasets used in our experiments.
    See \Cref{tab:dataset-sources} in \Cref{sec:dataset-sources} for references on these datasets.
  }
  \label{tab:real-data}
  \begin{tabular}{
      l
      S[table-format=5.0,round-mode=off]
      S[table-format=7.0,round-mode=off]
      S[table-format=1.5,round-mode=figures,round-precision=2]
    }
    \toprule
    Dataset            & {\(n\)} & {\(p\)} & {Density} \\ \midrule
    \dataset{bcTCGA}   & 536     & 17322   & 1         \\
    \dataset{news20}   & 19996   & 1355191 & 0.0003357 \\
    \dataset{rcv1}     & 20242   & 44504   & 0.00166   \\
    \dataset{Rhee2006} & 842     & 360     & 0.02469   \\ \bottomrule
  \end{tabular}
\end{table}

\Cref{fig:real-data} shows the suboptimality for the objective function $P$ as a function of the time for the four different datasets.
We see that when the regularization parameter is set at $\lambda_{\text{max}}/2$ and $\lambda_{\text{max}}/10$, our proposed solver is faster than all its competitors---especially when the datasets become larger.
This is even more visible for the \dataset{news20} dataset where we see that our proposed method is faster by at least one order of magnitude.

\begin{figure*}[!t]
  \centering
  \includegraphics{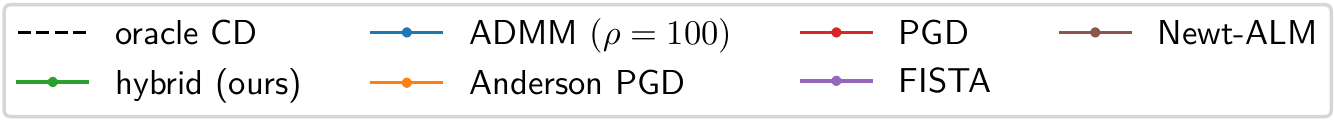}
  \includegraphics{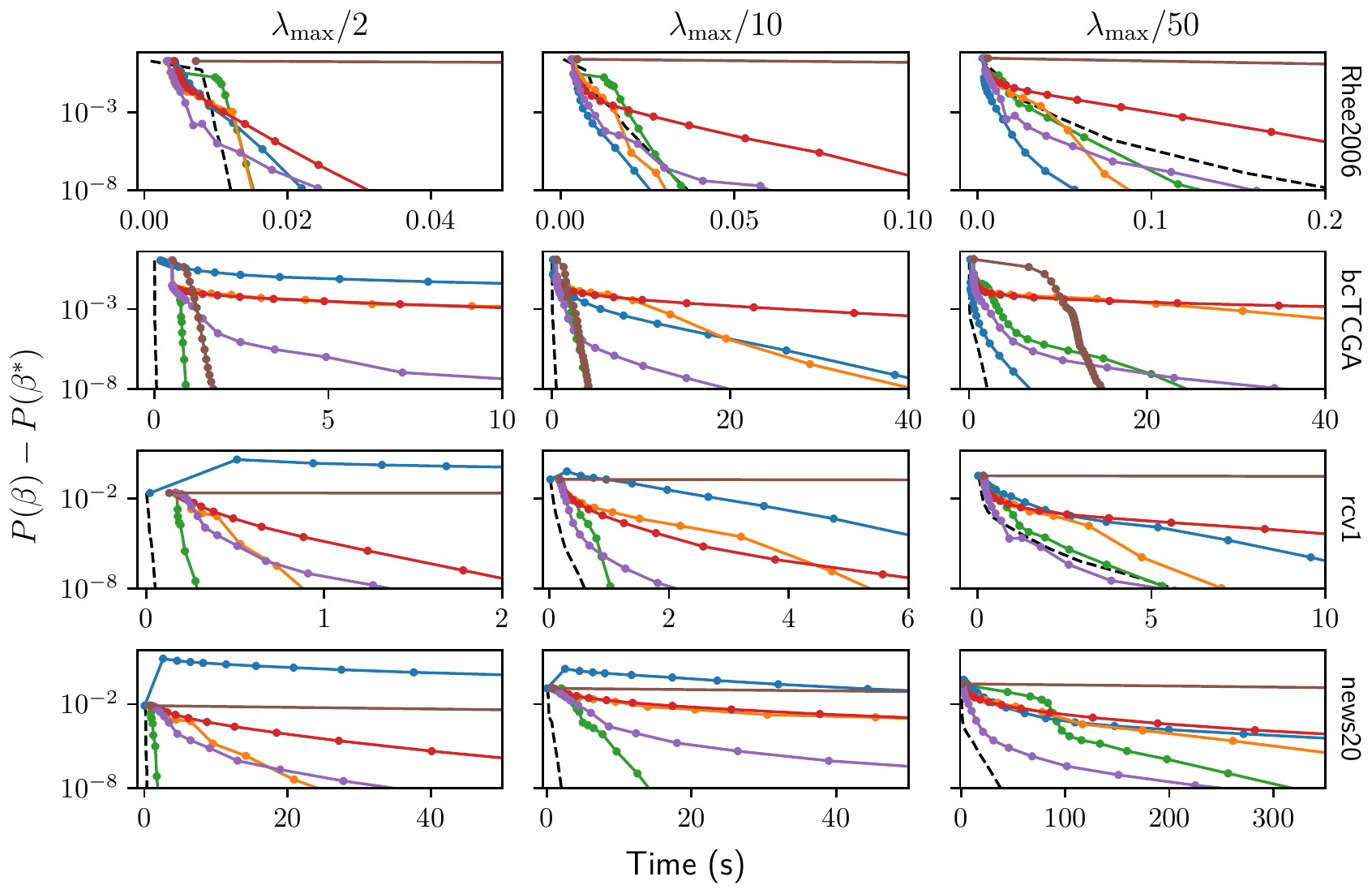}
  \caption{Benchmark on real datasets. The plots show suboptimality as a function of time for SLOPE on multiple simulated datasets and $\lambda$ sequences of varying strength.}
  \label{fig:real-data}
\end{figure*}

When the parametrization value is set to $\lambda_{\text{max}}/50$, our algorithm remains competitive on the different datasets.
It can be seen that the different competitors do not behave consistently across the datasets.
For example, the \texttt{Newt-ALM} method is very fast on the \dataset{bcTCGA} dataset but is very slow on the \dataset{news20} dataset whereas the \texttt{hybrid} method remains very efficient in both settings.

%% file: discussion.tex
\section{DISCUSSION}\label{sec:discussion}

In this paper we have presented a new, fast algorithm for solving Sorted L-One Penalized Estimation (SLOPE).
Our method relies on a combination of proximal gradient descent to identify the cluster structure of the solution and coordinate descent to allow the algorithm to take large steps.
In our results, we have shown that our method often outperforms all competitors by orders of magnitude for high-to-medium levels of regularization and typically performs among the best algorithms for low levels of regularization.

We have not, in this paper, considered using screening rules for SLOPE~\parencite{larsson2020c,elvira2022}.
Although screening rules work for any algorithm considered in this article, they are particularly effective when used in tandem with coordinate descent~\parencite{fercoq2015} and, in addition, easy to implement due to the nature of coordinate descent steps.
Coordinate descent is moreover especially well-adapted to fitting a path of \(\lambda\) sequences~\parencite{friedman2007,friedman2010}, which is standard practice during cross-validating to obtain an optimal \(\lambda\) sequence.

Future research directions may include investigating alternative strategies to split clusters, for instance by considering the directional derivatives with respect to the coefficients of an entire cluster at once.
Another potential approach could be to see if the full proximal gradient steps might be replaced with batch stochastic gradient descent in order to reduce the costs of these steps.
It would also be interesting to consider whether gap safe screening rules might be used not only to screen predictors, but also to deduce whether clusters are able to change further during optimization.
Finally, combining cluster identification of proximal gradient descent with solvers such as second order ones as in~\textcite{bareilles2022newton} is a direction of interest.

%% file: acknowledgements.tex
\subsubsection*{Acknowledgements}

The experiments presented in this paper were carried out using the HPC facilities of the University of Luxembourg~\parencite{Varette2022} (see \texttt{\href{http://hpc.uni.lu}{hpc.uni.lu}}).

The results shown here are in whole or part based upon data generated by the TCGA Research Network: \url{https://www.cancer.gov/tcga}.

%% file: proofs.tex
\section{PROOFS}
\label{sec:proofs}

\subsection{Proof of \Cref{thm:sl1-directional-derivative}}
\label{app:proof_directional_derivative}

Let \(c^{\setminus k}\) be the set containing all elements of $c$ except the $k$-th one: $c^{\setminus k} =  \{c_1, \ldots c_{k-1}, c_{k+1}, \ldots, c_m \}$.

From the observations in \Cref{rem:permutation_C_z},
we have the following cases to consider: \(|z| \in c^{\setminus k}\),
\(|z| = 0\), and \(|z| \notin \{0\} \cup c^{\setminus k}\).

Since \(C(z + \delta h) = C(z) = \cC_k\) and $\sign(z + \delta h) = \sign(z)$ for $h$ small enough,
\begin{align}
	H(z + \delta h) - H(z)
	 & = \sum_{j =1}^p |\beta(z + \delta h)_j| \lambda_{(j)^-_{z + \delta h}}
	- \sum_{j=1}^p |\beta(z)_j| \lambda_{(j)^-_z} \nonumber                                       \\
	 & = \sum_{j =1}^p (|\beta(z + \delta h)_j| - |\beta(z)_j|) \lambda_{(j)^-_z} \nonumber       \\
	 & = \sum_{j =1}^p (|\beta(z + \delta h)_j| - |\beta(z)_j|) \lambda_{(j)^-_z} \nonumber       \\
	 & = \sum_{j \in C(z)}^p (|\beta(z + \delta h)_j| - |\beta(z)_j|) \lambda_{(j)^-_z} \nonumber \\
	 & = \sum_{j \in C(z)} \sign(\beta(z)_j) (z + \delta h - z) \lambda_{(j)^-_z} \nonumber       \\
	 & = \sum_{j \in C(z)} \sign(z) \delta h  \lambda_{(j)^-_z} \nonumber                         \\
	 & = \sum_{j \in \cC_k} \sign(z) \delta h  \lambda_{(j)^-_z} \, .
\end{align}

\paragraph{Case 2}
Then if  $z \neq 0$ and $|z|$ is equal to one of the $c_i$'s, $i \neq k$,  one has $C(z) = \cC_k \cup \cC_i$, $C(z + \delta h) = \cC_k$, and $\sign(z + \delta h) = \sign(z)$ for $h$ small enough.
Thus
\begin{align}
	H(z + \delta h) - H(z)
	 & = \sum_{j =1}^p |\beta(z + \delta h)_j| \lambda_{(j)^-_{z + \delta h}}
	- \sum_{i=1}^p |\beta(z)_j| \lambda_{(j)^-_z}  \nonumber                                         \\
	 & = \sum_{j \in \cC_k \cup \cC_i} \left( |\beta(z + \delta h)_j| \lambda_{(j)^-_{z + \delta h}}
	- |\beta(z)_j| \lambda_{(j)^-_z} \right)  \nonumber                                              \\
	 & = \sum_{j \in \cC_k} \left( c_i + \delta h \right) \lambda_{(i)^-_{z + \delta h}}
	- c_i \lambda_{(i)^-_z}
	+ \sum_{j \in \cC_i} \left( c_i \lambda_{(j)^-_{z + \delta h}}
	- c_i \lambda_{(i)^-_z} \right) \, .
\end{align}
Note that there is an ambiguity in terms of permutation, since, due to the clustering, there can be more than one permutation reordering $\beta(z)$.
However, choosing any such permutation result in the same values for the computed sums.

\paragraph{Case 3} Finally let us treat the case $z = 0$.
If $c_m = 0$ then the proof proceeds as in case 2, with the exception that $|\beta(z + \delta h)| = h$ and so the result is just:
\begin{align}
	H(z + \delta h) - H(z)
	 & = h \sum_{j \in \cC_k} \lambda_{(i)^-_{z + \delta h}} \, .
\end{align}
If $c_m \neq 0$, then the computation proceeds exactly as in case 1.

\subsection{Proof of \Cref{thm:thresholding-operator}}

Recall that \(G(z) : \mathbb{R} \to \mathbb{R}\) is a convex,
continuous piecewise-differentiable function with breakpoints whenever \(|z| =
c_i^{\setminus k}\) or \(z = 0\). Let \(\gamma = c_k \norm{\tilde{x}}^2+ \tilde x^T r\)
and \(\omega = \norm{\tilde{x}}^2\) and note that the optimality criterion
for~\eqref{pb:cluster-problem} is
\[
	\delta(\omega z - \gamma) + H'(z; \delta) \geq 0, \quad
	\forall \delta \in \{-1, 1\},
\]
which is equivalent to
\begin{equation}
	\label{eq:optimality-inequality}
	\omega z - H'(z; -1) \leq \gamma \leq \omega z + H'(z; 1).
\end{equation}
We now proceed to show that there is a solution \(z^* \in \argmin_{z \in
	\mathbb{R}} H(z)\) for every interval over \(\gamma \in \mathbb{R}\).

First, assume that the first case in the definition of \(T\) holds
and note that this is equivalent to~\eqref{eq:optimality-inequality} with \(z
= 0\) since \(C({\varepsilon_c}) = C(-{\varepsilon_c})\) and
\(\lambda_{(j)^-_{-{\varepsilon_c}}} = \lambda_{(j)^-_{{\varepsilon_c}}}\).
This is sufficient for \(z^* = 0\).

Next, assume that the second case holds and observe that this is equivalent
to~\eqref{eq:optimality-inequality} with
\(z = c_i^{\setminus k}\), since
\(C(c_i + {\varepsilon_c}) = C(-c_i - {\varepsilon_c})\) and
\(C(-c_i + {\varepsilon_c}) = C(c_i - {\varepsilon_c})\). Thus \(z^* =
\sign(\gamma)c_i^{\setminus k}\).

For the third case, we have
\[
	\smashoperator{\sum_{j \in C(c_i + {\varepsilon_c})}} \lambda_{(j)^-_{c_i + {\varepsilon_c}}}
	=
	\smashoperator[r]{\sum_{j \in C(c_{i-1} - {\varepsilon_c})}} \lambda_{(j)^-_{c_{i-1} - {\varepsilon_c}}}
\]
and therefore~\eqref{eq:optimality-inequality} is equivalent to
\[
	c_i < \frac{1}{\omega} \bigg( |\gamma| - \smashoperator{\sum_{j \in C(c_i + {\varepsilon_c})}} \lambda_{(j)^-_{c_i + {\varepsilon_c}}} \bigg) < c_{i -1}.
\]
Now let
\begin{equation}
	\label{eq:differentiable-solution}
	z^* = \frac{\sign(\gamma)}{\omega} \bigg( |\gamma| - \smashoperator{\sum_{j \in C(c_i + {\varepsilon_c})}} \lambda_{(j)^-_{c_i + {\varepsilon_c}}} \bigg)
\end{equation}
and note that \(|z^*| \in \big(c_i^{\setminus k}, c_{i-1}^{\setminus k}\big)\) and hence
\[
	\frac{1}{\omega} \bigg( |\gamma| - \smashoperator{\sum_{j \in C(c_i + {\varepsilon_c})}} \lambda_{(j)^-_{c_i + {\varepsilon_c}}} \bigg)
	=
	\frac{1}{\omega} \bigg( |\gamma| - \smashoperator{\sum_{j \in C(z^*)}} \lambda_{(j)^-_{z^*}} \bigg).
\]
Furthermore, since \(G\) is differentiable in \(\big(c_i^{\setminus k}, c_{i-1}^{\setminus k}\big)\), we have
\[
	\frac{\partial}{\partial z} G(z) \Big|_{z = z^*}
	= \omega z^* - \gamma + \sign(z^*) \smashoperator{\sum_{j \in C(z^*)}} \lambda_{(j)^-_{z^*}} = 0,
\]
and therefore~\eqref{eq:differentiable-solution} must be the solution.

The solution for the last case follows using reasoning analogous to that of the
third case.

\subsection{Proof of \Cref{lem:convergence}}





To prove the lemma, we will show that $\lim_{t \rightarrow \infty} \beta^{(t)} \in \Omega = \{\beta:0 \in \partial P(\beta)\}$ using Convergence Theorem A in \textcite[p.~91]{zangwill1969}.
For simplicity, we assume that the point to set map $A$ is generated by $v$ iterations of \Cref{alg:hybrid}, that is $A(\beta^{(0)})= \{\beta^{(vi)}\}_{i=0}^\infty$.
To be able to use the theorem, we need the following assumptions to hold.
\begin{enumerate}
	\item The set of iterates, $A(\beta^{(0)})$ is in a compact set.
	\item $P$ is continuous and if $\beta \notin \Omega = \{\beta:0 \in \partial P(\beta)\}$, then for any $\hat{\beta} \in A(\beta)$  it holds that $P(\hat{\beta}) < P(\beta)$.
	\item If $\beta  \in \Omega =\{\beta:0 \in \partial P(\beta)\}$, then for any $\hat{\beta}\in A(\beta)$ it holds that $P(\hat{\beta}) \leq P(\beta)$.
\end{enumerate}

Before tackling these three assumptions, we decompose the map into two parts: $v-1$ coordinate descent steps, $T_\text{CD}$, and one proximal gradient decent step, $T_{\text{PGD}}$.
This clearly means that
$$
	P(T_\text{CD}(\beta)) \leq P(\beta)
$$
for all $\beta \in \mathbb{R}^p$.
For $T_\text{PGD}$, we have two useful properties: first, if $||T_\text{PGD}(\beta) - \beta||=0$, then by Lemma~2.2 in \textcite{beck2009} it follows that $\beta \in \Omega$.
Second, by Lemma~2.3 in \textcite{beck2009}, using $x=y$, it follows that
$$
	P(T_\text{PGD}(\beta)) - P(\beta) \leq  - \frac{L(f)}{2}||T_\text{PGD}(\beta) - \beta||^2,
$$
where $L(f)$ is the Lipschitz constant of the gradient of $f(\beta)= \frac{1}{2}||y-X\beta||^2$.

We are now ready to prove that the three assumptions hold.
\begin{itemize}
	\item Assumption 1 follows from the fact that the level sets of $P$ are compact and from $P(T_{PGD}(\beta)) \leq P(\beta)$ and $P(T_{CD}(\beta)) \leq P(\beta)$.
	\item Assumption 2 holds since if $\beta \notin \Omega$, it follows that $||T_\text{PGD}(\beta) - \beta|| > 0$ and thus $P(T_\text{PGD}(\beta)) < P(\beta)$.
	\item Assumption 3 follows  from $P(T_{PGD}(\beta)) \leq P(\beta)$ and $P(T_{CD}(\beta)) \leq P(\beta)$.
\end{itemize}
Using Theorem~1 from \textcite{zangwill1969}, this means that \Cref{alg:hybrid} converges as stated in the lemma.

\subsection{Partial Smoothness of the Sorted $\ell_1$ Norm}
\label{app:sec:partly_smooth}

In this section, we prove that the sorted $\ell_1$ norm $J$ is partly smooth~\parencite{lewis2002a}.
This allows us to apply results about the structure identification of the proximal gradient algorithm.

\begin{definition}
	Let $J$ be a proper closed convex function and $x$ a point of its domain such that $\partial J(x) \neq \emptyset$.
	$J$ is said to be partly smooth at $x$ relative to a set $\cM$ containing $x$ if:
	\begin{enumerate}
		\item $\cM$ is a $C^2$-manifold around $x$ and $J$ restricted to $\cM$ is $C^2$ around $x$.
		\item The tangent space of $\cM$ at $x$ is the orthogonal of the parallel space of $\partial J(x)$.
		\item $\partial J$ is continuous at $x$ relative to $\cM$.
	\end{enumerate}
\end{definition}

Because the sorted \(\ell_1\) norm is a polyhedral, it follows immediately that it
is partly smooth~\parencite[Example 18]{vaiter2017}. But since we believe a direct
proof is interesting in and of itself, we provide and prove the following proposition here.

\begin{proposition}
	Suppose that the regularization parameter $\lambda$ is a strictly decreasing sequence. Then the sorted $\ell_1$ norm is partly smooth at any point of $\bbR^p$.
\end{proposition}

\begin{proof}

	Let $m$ be the number of clusters of $x$ and $\cC_1, \ldots, \cC_m$ be those clusters, and let $c_1 > \cdots > c_m > 0$ be the  value of  $\lvert x \rvert$ on the clusters.

	We define $\varepsilon_c$ as in \Cref{eq:epsilon-c} and
	let $\cB = \{u \in \bbR^p: \lVert u - x \rVert_\infty < \varepsilon_c / 2\}$.
	Let $v_k \in \bbR^p$ for $k \in [m]$  be equal to $\sign(x_{\cC_k})$ on $\cC_k$ and to 0 outside, such that $x = \sum_{k=1}^m c_k v_k$.
	We define
	\begin{equation*}
		\cM =
		\begin{cases}
			\Span(v_1, \ldots, v_m) \cap \cB \,     & \text{if } c_m \neq 0 \, , \\
			\Span(v_1, \ldots, v_{m-1}) \cap \cB \, & \text{otherwise} \, .
		\end{cases}
	\end{equation*}
	We will show that $J$ is partly smooth at $x$ relative to $\cM$.

	As a first statement, we prove that any $u\in\cM$ shares the same clusters as $x$. For any $u\in\cM$ there exists $c' \in \bbR^m$, $u = \sum_{k=1}^m c'_k v_k$ (with $c'_m = 0$ if $c_m = 0$).
	Suppose that there exist $k \neq k'$ such that $c'_k = c'_{k'}$.
	Then since $\lVert x - u \rVert_\infty = \max_k |c_k - c'_k|$ and $|c_k - c_{k'}| > \varepsilon_c$, one has:
	\begin{align*}
		\varepsilon_c < |c_k - c_{k'}|
		 & = |c_k - c'_k + c'_{k'} - c_{k'}|      \\
		 & \leq |c_k - c'_k| + |c'_{k'} - c_{k'}| \\
		 & \leq 2 \lVert x - u \rVert_\infty      \\
		 & \leq  \varepsilon_c \, .
	\end{align*}

	This shows that clusters of any $u \in \cM$ are equal to clusters of $x$.
	Further, clearly the tangent space of $\cM$ at $x$ is $\Span(v_1, \hdots, v_m)$ if $c_m \neq 0$ and $\Span(v_1, \ldots, v_{m-1})$ otherwise.




	\begin{enumerate}
		\item The set $\cM$ is then the intersection of a linear subspace and an open ball, and hence is a $\cC^2$ manifold.
		      Since the clusters of any $u\in\cM$ are the same as the clusters of $x$, we can write that
		      \begin{align}{}
			      J(u) = \sum_{k=1}^m \left( \sum_{j \in \cC_k} \lambda_j \right)c_k' \enspace ,
		      \end{align}
		      and hence $J$ is linear on $\cM$ and thus $\cC^2$ around $x$.
		\item We let $x_\downarrow$ denote a version of \(x\) sorted in non-increasing order and let $R:\bbR^p \rightarrow \mathbb{N}^p$ be the function that returns the ranks of the absolute values of its argument. The subdifferential of $J$ at $x$ \parencite[Thm. 1]{larsson2020c}\footnote{We believe there to be a typo in the definition of the subgradient in \parencite[Thm. 1]{larsson2020c}. We believe the argument of \(R\) should be \(g\), not \(s\), since otherwise there is a dimension mismatch.} is the set of all $g\in\bbR^p$ such that
		      \begin{align}\label{eq:slope_subdiff}
			      g_{\cC_i} \in \cG_i \triangleq \left \{ s \in \bbR^{|C_i|} :
			      \begin{cases}
				      \cumsum(|s|_{\downarrow} - \lambda_{R(g)_{\cC_i}}) \preceq 0 & \text{if } x_{\cC_i} = \textbf{0} \, , \\
				      \cumsum(|s|_{\downarrow} - \lambda_{R(g)_{\cC_i}}) \preceq 0                                          \\
				      \quad \text{ and } \sum_{j\in \cC_i} (|s_j| - \lambda_{R(g)_{\cC_i}}) = 0                             \\
				      \quad \text{ and } \sign(x_{\cC_i}) =  \sign(s)              & \mathrm{otherwise.}
			      \end{cases}
			      \right \}
		      \end{align}
		      Hence, the problem can be decomposed over clusters.
		      We will restrict the analysis to a single $\cC_i$ without loss of generality and proceed in $\bbR^{|\cC_i|}$.
		      \begin{itemize}

			      \item First we treat the case where $|\cC_i|=1$  and $x_{\cC_i}\neq 0$.
			            The set $\cG_i$ is then the singleton $\{\sign(x_{\cC_i})\lambda_{R(s)_{\cC_i}}\}$ and its parallel space is simply $\{0\}$.
			            Hence, $\parset(\cG_i)^\perp = \bbR = \Span (\sign(x)_{\cC_i})$.
			      \item Then, we study the case where $|\cC_i|\neq 1$  and $x_{\cC_i}\neq \textbf{0}$.
			            Since for all $j \in [p]$, $\lambda_j\neq 0$ and $\lambda$ is a strictly decreasing sequence, we have that for $\varepsilon > 0$ small enough, the $|\cC_i| -1$ points $\lambda_{R(g)_{\cC_i}} + \varepsilon [-\sign(x_{\cC_i})_1, \sign(x_{\cC_i})_2, 0, \hdots, 0]^T$, $\lambda_{R(g)_{\cC_i}} + \varepsilon [0, -\sign(x_{\cC_i})_2, \sign(x_{\cC_i})_3, \hdots, 0]^T$, $\hdots$, $\lambda_{R(g)_{\cC_i}} + \varepsilon [0, 0, 0, \hdots,-\sign(x_{\cC_i})_{|\cC_{i}|-1}, \sign(x_{\cC_i})_{|\cC_{i}|}]^T$ belong to $\cG_i$.
			            Since these vectors are linearly independent, and using the last equality in the feasible set that, we have that
			            \begin{align}
				            \sum_{j\in \cC_i} \sign(x_j)s_j = \sum_{j \in \cC_i}  \lambda_{R(g)_{\cC_i}} \nonumber \enspace .
			            \end{align}
			            Its parallel space is simply the set $\{s\in\bbR^{|\cC_i|} : \sum_{j\in \cC_i} \sign(x_j)s_j = 0\}$, that is just $\Span(\sign(x_{\cC_i}))^\perp$.
			            Hence $\parset(\cG_i)^\perp = \Span(\sign(x_{\cC_i}))$.

			      \item  Finally, we study the case where $x_{\cC_m} = \textbf{0}$.
			            Then the $\ell_\infty$ ball $\{s \in \bbR^{|\cC_m|}: \Vert s \Vert_\infty \leq \lambda_p\}$ is contained in the feasible set of the differential, hence the parallel space of $\cG_m$ is $\bbR^{|\cC_m|}$ and its orthogonal is reduced to $\{ \mathbf{0} \}$.
		      \end{itemize}

		      We can now prove that $\parset(\partial J(x))^\perp$ is the tangent space of $\cM$.
		      From the decomposability of $\partial J$ (\Cref{eq:slope_subdiff}), one has that $u \in \parset(\partial J(x)) ^\perp$ if and only if $u_{\cC_i} \in \parset (\cG_i)^\perp$ for all $i \in [m]$.

		      If $c_m > 0$, we have
		      \begin{equation}
			      \begin{aligned}
				      \parset(\partial J(x))^\perp & = \{ u \in \bbR^p : \forall i \in [m], u_{\cC_i} \in \parset (\cG_i)^\perp \}   \\
				                                   & = \{ u \in \bbR^p : \forall i \in [m], u_{\cC_i} \in \Span(\sign(x_{\cC_i})) \} \\
				                                   & = \Span(v_1, \ldots, v_m) \, .
			      \end{aligned}
		      \end{equation}

		      If $c_m=0$, we have
		      \begin{equation}
			      \begin{aligned}
				      \parset(\partial J(x))^\perp & = \{ u \in \bbR^p : \forall i \in [m], u_{\cC_i} \in \parset (\cG_i)^\perp \}                                             \\
				                                   & = \{ u \in \bbR^p : \forall i \in [m- 1], u_{\cC_i} \in \Span(\sign(x_{\cC_i}))  \quad  \& \quad u_{\cC_m} = \mathbf{0}\} \\
				                                   & = \Span(v_1, \ldots, v_{m-1}) \, .
			      \end{aligned}
		      \end{equation}
		\item The subdifferential of $J$ is a constant set locally around $x$ along $\cM$ since the clusters of any point in the neighborhood of $x$ in $\cM$ shares the same clusters with $x$.
		      This shows that it is continuous at $x$ relative to $\cM$.
	\end{enumerate}

\end{proof}

\begin{remark}
	We believe that the assumption $\lambda_1 > \cdots > \lambda_p$ can be lifted, since for example the $\ell_1$ and $\ell_\infty$ norms are particular instances of $J$ that violate this assumption, yet are still partly smooth.
	Hence this assumption could probably be lifted in a future work using a slightly different proof.
\end{remark}

%% file: additional_experiments.tex
\section{ADDITIONAL EXPERIMENTS}\label{sec:add_expes}

\subsection{\pkg{glmnet} versus \pkg{SLOPE} Comparison}
\label{sec:slope-vs-glmnet}

In this experiment, we ran the \pkg{glmnet}~\parencite{friedman2022} and \pkg{SLOPE}~\parencite{larsson2022d} packages on the \dataset{bcTCGA} dataset, selecting the regularization sequence \(\lambda\) such that there were 100 nonzero coefficients and clusters at the optimum for \pkg{glmnet} and \pkg{SLOPE} respectively.
We used a duality gap of \(10^{-6}\) as stopping criteria.
The features were centered by their means and scaled by their standard deviation.
The code is available at \href{https://github.com/jolars/slopecd}{\url{github.com/jolars/slopecd}}.

\subsection{Study on Proximal Gradient Descent Frequency}
\label{sec:pgd-freq-study}

To study the impact of the frequence at which the PGD step in the \texttt{hybrid} solver is used, we performed a comparative study with the \dataset{rcv1} dataset.
We set this parameter to values ranging from $1$ \textit{i.e.}, the \texttt{PGD} algorithm, to 9 meaning that a PGD step is taken every $9$ epochs.
The sequence of $\lambda$ has been set with the Benjamini-Hochberg method and parametrized with $0.1 \lambda_{\text{max}}$.

\Cref{fig:pgd_freq} shows the suboptimality score as a function of the time for the different values of the parameter controlling the frequency at which a PGD step is going to be taken.
A first observation is that as long as this parameter is greater than $1$ meaning that we perform some coordinate descent steps, we observe a significant speed-up.
For all our experiments, this parameter was set to $5$.
The figure also shows that any choice between $3$ and $9$ would lead to similar performance for this example.

\begin{figure*}[htb]
  \centering
    \includegraphics{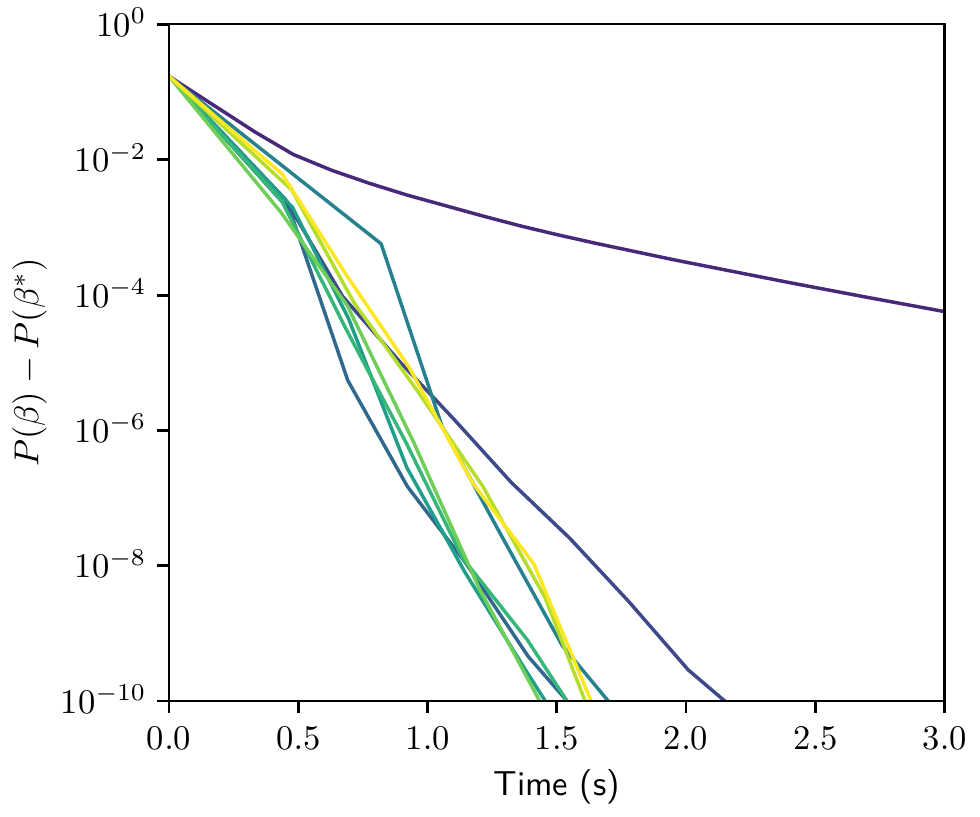}%
    \raisebox{0.585\height}{\includegraphics{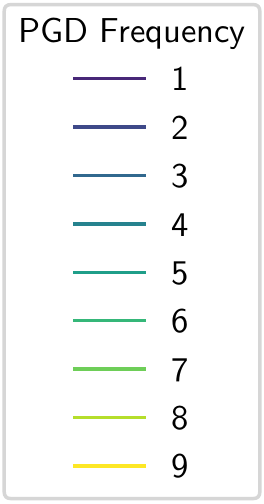}}
    \caption{Suboptimality score as a function of the time for different frequencies of the PDG step inside the \texttt{hybrid} solver for the \dataset{rcv1} dataset}
  \label{fig:pgd_freq}
\end{figure*}

\subsection{Benchmark with Different Parameters for the ADMM Solver}
\label{sec:admm-benchmarks}

We reproduced the benchmarks setting described in \Cref{sec:experiments} for the simulated and real data.
We compared the \texttt{ADMM} solver with our \texttt{hybrid} algorithm for different values of the augmented Lagrangian parameter $\rho$.
We tested three different values $10, 100$ and $1000$ as well as the adaptive method~\parencite[Sec. 3.4.1]{boyd2010}.

We present in \Cref{fig:simulated_appendix} and \Cref{fig:real_appendix} the suboptimality score as a function the time for the different solvers.
We see that the best value for $\rho$ depends on the dataset and the regularization strengh.
The value chosen for the main benchmark (\Cref{sec:experiments}) performs well in comparison to other \texttt{ADMM} solvers.
Nevertheless, our \texttt{hybrid} approach is consistently faster than the different  \texttt{ADMM} solvers.

\begin{figure*}[!t]
  \centering
  \includegraphics{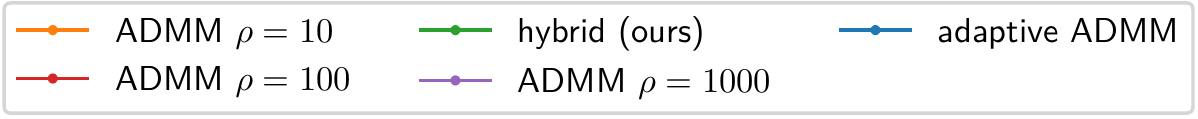}
  \includegraphics{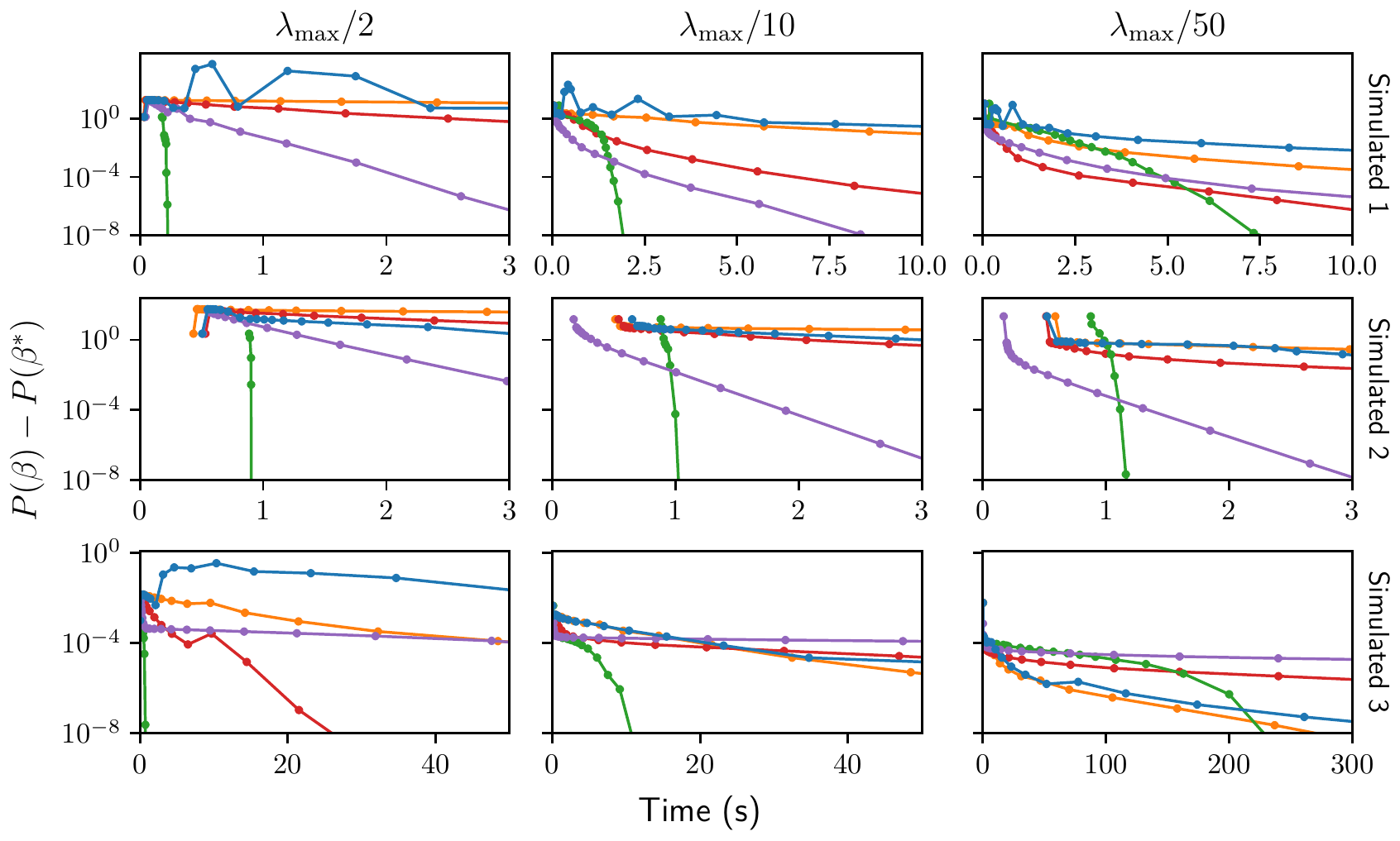}
  \caption{\textbf{Benchmark on simulated datasets.} Suboptimality score as a function of time for SLOPE on multiple simulated datasets and for multiple sequence of $\lambda$.}
  \label{fig:simulated_appendix}
\end{figure*}

\begin{figure*}[!t]
  \centering
  \includegraphics{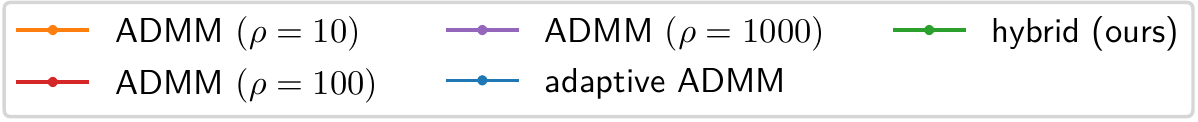}
  \includegraphics{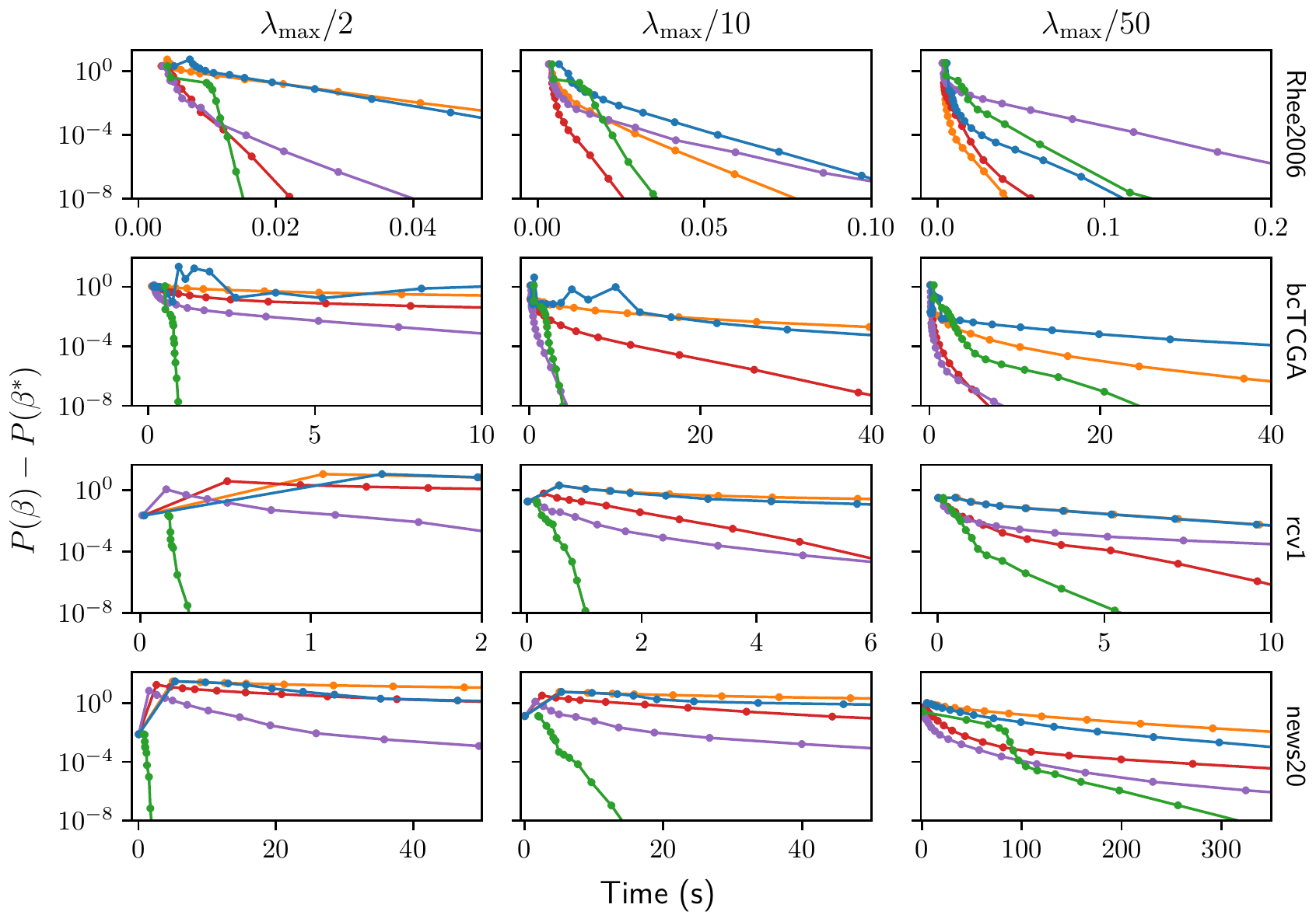}
  \caption{\textbf{Benchmark on simulated datasets.} Suboptimality score as a function of time for SLOPE on multiple simulated datasets and for multiple sequence of $\lambda$.}
  \label{fig:real_appendix}
\end{figure*}

%% file: other_datafits.tex
\section{EXTENSIONS TO OTHER DATAFITS}
\label{sec:other-datafits}

Our algorithm straightforwardly generalizes to problems where the quadratic datafit $\frac{1}{2} \lVert y - X \beta \rVert^2$ is replaced by $F(\beta) = \sum_{i = 1}^n f_i (X_{i:}^\top \beta)$, where the $f_i$'s are $L$ smooth (and so $F$ is $L * \lVert X \rVert_2^2$-smooth), such as logistic regression.

In that case, one has by the descent lemma applied to $F(\beta(z))$, using $F(\beta) = F(\beta(c_k))$,
\begin{equation}
  F(\beta(z)) + H(z) \leq F(\beta) + \sum_{j \in \cC_k} \nabla_j F(\beta) \sign \beta_j (z - c_k) + \frac{L \lVert \tilde x \rVert^2}{2} (z - c_k)^2 + H(z)
\end{equation}
and so a majorization-minimization approach can be used, by minimizing the right-hand side instead of directly minimizing $F(\beta(z)) + H(z)$.
Minimizing the RHS, up to rearranging, is of the form of \Cref{pb:cluster-problem}.

%% file: solver_details.tex
\section{IMPLEMENTATION DETAILS OF SOLVERS}
\label{sec:solver-details}

\subsection{ADMM}

Our implementation of the solver is based on \textcite{boyd2011}.
For high-dimensional sparse \(X\), we use the numerical LSQR algorithm~\parencite{paige1982} instead of the typical direct linear system solver.
We originally implemented the solver using the adaptive step size (\(\rho\)) scheme from \textcite{boyd2010} but discovered that it performed poorly.
Instead, we used \(\rho = 100\) and have provided benchmarks of the alternative configurations in \Cref{sec:admm-benchmarks}.

\subsection{Newt-ALM}

The implementation of the solver is based on the pseudo-code provided in \textcite{Ziyan2019}.
According to the authors' suggestions, we use the Matrix inversion lemma for high-dimensional and sparse \(X\) and the preconditioned conjugate gradient method if, in addition, \(n\) is large.
Please see the source code for further details regarding hyper-parameter choices for the algorithm.

After having completed our own implementation of the algorithm, we received an implementation directly from the authors.
Since our own implementation performed better, however, we opted to use it instead.

%% file: dataset_sources.tex
\section{REFERENCES AND SOURCES FOR DATASETS}
\label{sec:dataset-sources}

In \Cref{tab:dataset-sources}, we list the reference and source (from which the data was gathered) for each of the real datasets used in our experiments.

\begin{table}[hbt]
  \centering
  \caption{Sources and references for the real data sets used in our experiments.\label{tab:dataset-sources}}
  \begin{tabular}{lll}
    \toprule
    Dataset            & Reference                              & Source                 \\
    \midrule
    \dataset{bcTCGA}   & \textcite{nationalcancerinstitute2022} & \textcite{breheny2022} \\
    \dataset{news20}   & \textcite{keerthi2005}                 & \textcite{chang2016}   \\
    \dataset{rcv1}     & \textcite{lewis2004}                   & \textcite{chang2016}   \\
    \dataset{Rhee2006} & \textcite{rhee2006}                    & \textcite{breheny2022} \\
    \bottomrule
  \end{tabular}
\end{table}